\documentclass[11pt]{amsart}
\usepackage{amssymb,amsthm,amsmath,amstext}
\usepackage{mathrsfs}  % allows nice script font for sheaves
\usepackage{mathtools} % allows more extendible arrows
\usepackage[usenames,dvipsnames,svgnames,table]{xcolor}
\usepackage{graphicx}
\usepackage[all]{xy}
\usepackage{enumitem}
\usepackage[left=1in,top=1in,right=01in,bottom=1.4in]{geometry}

\numberwithin{equation}{section}

\theoremstyle{plain}
\newtheorem{prop}{Proposition}
\newtheorem{theo}[prop]{Theorem}
\newtheorem*{theo*}{Theorem}
\newtheorem{coro}[prop]{Corollary}

\newtheorem{lemm}[prop]{Lemma}

\theoremstyle{definition}
\newtheorem{defi}[prop]{Definition}

\newtheorem{ques}[prop]{Question}

\newtheorem{rema}[prop]{Remark}

\newcommand{\Z}{\mathbb Z}

\renewcommand{\P}{\mathbb P}

 \newcommand{\PP}{\mathbb{P}}

\begin{document}
\date{\today}

\title[Reconstructing function fields from Milnor K-theory]{Reconstructing function fields from Milnor K-theory}

\author[Anna Cadoret]{Anna Cadoret}
\address{\hspace{-0.5cm}Anna Cadoret\newline
IMJ-PRG -- Sorbonne Universit\'e, Paris, FRANCE\newline
\textit{anna.cadoret@imj-prg.fr}}

\author[Alena Pirutka]{Alena Pirutka}
\address{\hspace{-0.5cm}Alena Pirutka \newline
CIMS,
New York University, 
New York, U.S.A.\newline
National Research University Higher School of Economics, RUSSIAN FEDERATION\newline
\textit{pirutka@cims.nyu.edu}}

\maketitle

\begin{abstract}
 Let $F$ be a finitely generated regular field extension  of transcendence degree  $\geq 2$  over a  perfect field   $k$. We show that the multiplicative group $F^\times/k^\times$ endowed with the equivalence relation induced by algebraic dependence on $k$ determines the isomorphism class of  $F$ in a functorial way. As a special case of this result, we obtain that the isomorphism class of the graded Milnor $K$-ring $K^M_*(F)$ determines the isomorphism class of $F$, when $k$ is algebraically closed or finite. 
\end{abstract}

\vspace{.6cm}

\section{Introduction} 

\noindent This paper is motivated  by the following general question, for which no counter-example seems to be known. \\

\begin{ques}\label{Q}  Does the Milnor K-ring $K_*^M(F)$ determine the isomorphism class of the field $F$?\end{ques}

\noindent One may also ask whether or not the above holds in a functorial way, that is, whether or not the Milnor K-ring functor is fully faithfull from the groupoid of fields to the groupoid of $\Z_{\geq 0}$-graded rings. This naive functorial version does not hold in general as shown by the example of finite fields $F$ where $K_*^M(F)$ has extra  automorphisms induced by $x\rightarrow x^u$ for $u$ prime to $|F|-1$ on $K_1^M(F)$. To get a viable functorial version of Question \ref{Q}, one should at least kill such extra automorphisms. \\

 \noindent Since $K_1^M(F)=F^\times$,  Question \ref{Q} essentially reduces to reconstructing the additive structure of $F$ from the multiplicative group $F^\times$ endowed with additional data that can be detected by the Milnor K-ring. Our main result (Theorem \ref{TAD}) asserts that for  a finitely generated regular field extension $F$ of transcendence degree  $\geq 2$  over a  perfect field $k$, the multiplicative group $F^\times/k^\times$ endowed with the equivalence relation induced by algebraic dependence on $k$ determines the isomorphism class of  $F$ in a functorial way.  In Section \ref{Recons}, we show (Theorem \ref{Div}) that for  a finitely generated regular field extension  of  a field $k$ which is either  algebraically closed or  finite, the Milnor K-ring detects algebraic dependence. This is a consequence of deep K-theoretic results - the $n=2$ case of the Bloch-Kato conjecture \cite{MS} when $k$ is algebraically closed and of the Bass-Tate conjecture \cite{Tate} when $k$ is finite. Combined with Theorem \ref{TAD}, this enables us show that the Milnor K-ring modulo the ideal of divisible elements (resp. of torsion elements) determines in a functorial way finitely generated regular field extensions of  transcendence degree $\geq 2$ over  algebraically closed fields   (resp. over finite fields) (see Corollary \ref{FullMT}). In particular, this provides a purely K-theoretic description of the group of birational automorphisms of   normal  projective varieties of dimension $\geq 2$ over algebraically closed or finite fields (Corollary \ref{Auto}).\\

\subsection{Main Result}  Recall that a field extension $F/k$ is {\it regular} if $k$ is algebraically closed in $F$. Let $F/k$ be a regular  field extension.

\noindent 
\begin{defi}\label{AlgIndep}  We say that  $\overline{x},\overline{y}\in F^\times/k^\times $ are algebraically dependent and write $\overline{x}\equiv  \overline{y} $ if  either $\overline{x}=1=\overline{y}$ or  $\overline{x}\neq 1\neq \overline{y}$ and some (equivalently, all) lifts $x,y\in F^\times$ of $\overline{x},\overline{y}\in F^\times/k^\times$ are algebraically dependent over $k$. The relation $\equiv $ is an equivalence relation on $F^\times/k^\times $.
\end{defi}

\noindent  Let $F/k$, $F'/k'$ be regular  field extensions.

\begin{defi}
 We say that a group morphism $\overline{\psi}: F^\times/k^\times \rightarrow F'^{\,\times}/k '^{\,\times}$ {\it preserves
 algebraic dependence} if  for every $\overline{x},\overline{y}\in F^\times/k^\times$ the following holds:  $\overline{x}\equiv  \overline{y} $ if and only if $\overline{\psi}(\overline{x})\equiv\overline{\psi}(\overline{y}) $.
\end{defi}
\noindent (In particular, a group morphism preserving algebraic dependence is automatically injective).\\

\noindent For a subfield $E\subset F$, write $\overline{E^F}\subset F$ for the algebraic closure of $E$ in $F$. Then a group morphism $\overline{\psi}: F^\times/k^\times \rightarrow F'^{\,\times}/k '^{\,\times}$   preserves
 algebraic dependence if and only if $\overline{k(x)^F}^\times/k^\times =\overline{\psi}^{-1}(\overline{k'(\psi(x))^{F'}}/k'^{\,\times})$ for every $x\in F$ and some (equivalently, every) lift $\psi(x)\in F'$ of $\overline{\psi}(\overline{x})$.\\

 \noindent  Let $\hbox{\rm Isom}(F,F')$  denote the set of field isomorphisms $F\tilde{\rightarrow} F'$ and  $$\hbox{\rm Isom}(F/k,F'/k')\subset \hbox{\rm Isom}(F,F') $$ denote the subset of isomorphisms $F\tilde{\rightarrow} F'$  inducing field isomorphisms $k\tilde{\rightarrow} k'$. \\

 \noindent Let $
\hbox{\rm Isom}(F^\times/k^\times, F'^{\,\times}/k '^{\,\times})
 $ denote the set of group isomorphisms $F^\times/k^\times\tilde{\rightarrow} F'^{\,\times}/k '^{\,\times}$ and $$\hbox{\rm Isom}^{\equiv}(F^\times/k^\times, F'^{\,\times}/k '^{\,\times})\subset \hbox{\rm Isom}(F^\times/k^\times, F'^{\,\times}/k '^{\,\times})$$ 
 the subset of isomorphisms $F^\times/k^\times\tilde{\rightarrow} F'^{\,\times}/k '^{\,\times}$ preserving algebraic dependence. The group  $\Z/2$ acts on the set $\hbox{\rm Isom}^{\equiv}(F^\times/k^\times, F'^{\,\times}/k '^{\,\times})$ by $\overline{\psi}\rightarrow \overline{\psi}^{\,-1}$. Write 
 $$\overline{\hbox{\rm Isom}}^{\equiv}(F^\times/k^\times, F'^{\,\times}/k '^{\,\times}) $$
for the resulting quotient.\\

\begin{theo}\label{TAD}
 Let $k,k'$ be  perfect  fields  of characteristic $p\geq 0$ and let $F/k$, $F'/k'$ be finitely generated regular field extensions of transcendence degree $\geq 2$.  Then the canonical map
$$ \hbox{\rm Isom}(F /k , F'/k ')\rightarrow \overline{\hbox{\rm Isom}}^{\equiv}(F^\times/k^\times, F'^{\,\times}/k '^{\,\times})$$
is bijective.
\end{theo} 

\subsection{Comparison with existing results}
Question \ref{Q} was  considered by Bogomolov and Tschinkel in \cite{BTM}, where they prove  (a variant of) Theorem \ref{TAD} for finitely generated regular extensions of characteristic $0$ fields (\cite[Thm. 2]{BTM})  and deduce from it    Corollary \ref{FullMT} for finitely generated field extensions of algebraically closed fields of characteristic $0$ (\cite[Thm. 4]{BTM}). \\

\noindent Variants of our results were also obtained by Topaz  from a smaller amount of $K$-theoretic information - mod-$\ell$ Milnor $K$-rings  (for finitely generated field extensions of transcendence degree   $\geq 5$ over algebraically closed field of characteristic $p\not= \ell$ \cite[Thm. B]{Topaz}) and  rational Milnor $K$-rings  (for finitely generated field extensions of transcendence degree   $\geq 2$ over algebraically closed field of characteristic $0$ \cite[Thm. 6.1]{Topaz1}) 
 but enriched with the additional data of  the so-called ``rational quotients" of $F/k$.  See also \cite[Rem. 6.2]{Topaz1} for some cases where the additional data of rational quotients can be removed. \\
 
 \noindent  Our strategy follows the one of Bogomolov and Tschinkel in \cite{BTM}, where the key idea is to parametrize lines in $F^\times/k^\times$ as intersections of multiplicatively shifted (infinite dimensional) projective subspaces of a specific form  arising from relatively algebraically closed subextensions of transcendence degree $1$. See Subsection \ref{StrategyIntro} for details. The strategy of Topaz is more sophisticated and goes through the reconstruction of the quasi-divisorial valuations of $F$ \textit{via} avatars of the theory of commuting-liftable pairs as developped in the framework of birational anabelian geometry.  Though not explicitly stated in the literature, it is likely that Theorem \ref{TAD} and  Corollary \ref{FullMT}  for finitely generated field extensions algebraically closed fields of characteristic $p>0$ could also be recovered from the  techniques of birational anabelian geometry as developed by Bogomolov-Tschinkel \cite{BTIntro}, Pop (\textit{e.g.} \cite{PopDG}, \cite{Pop}) and Topaz.\\
 
 \noindent To our knowledge, Theorem \ref{TAD} for finitely generated regular extensions of perfect fields of characteristic $p>0$ and  Corollary \ref{FullMT} for finitely generated field extensions of finite fields are new. \\

\subsection{Strategy of proof}\label{StrategyIntro} For simplicity, write $F^p\subset F$ for the subfield generated by $k$ and the $x^p$, $x\in F$  and $F^\times/p:=F^\times/F^{p\times}$.\\

\noindent The proof of Theorem \ref{TAD} is carried out in Section 3.  According to the fundamental theorem of projective geometry (see Lemma \ref{FTPG}, for which we give a self-contained proof in the setting of possibly infinite field extensions), it would be enough to show that a group isomorphism $\overline{\psi} :F^\times/k^\times\tilde{\rightarrow} F'^{\,\times}/k '^{\,\times}$ preserving algebraic dependence induces a bijection from  lines in $F^\times/k^\times$ to  lines in $F'^{\,\times}/k'^{\,\times}$. This would reduce the problem to describing lines in $F^\times/k^\times$ using only $\equiv$ and the multiplicative structure of $F^\times/k^\times$. This classical approach works well if $p=0$. The key observation of Bogomolov and Tschinkel in  \cite{BTM} is that every line can be multiplicatively  shifted to a line passing through a ``good" pair of points  and that those lines can be uniquely parametrized as intersections of  multiplicatively shifted (infinite dimensional) projective subspaces of a specific form arising from relatively algebraically closed subextensions of transcendence degree $1$ \cite[Thm. 22]{BTM}. This is the output of elaborate computations in \cite{BTM}. Later, Rovinsky suggested an alternative argument using differential forms; this is sketched in \cite[Prop. 9]{BTIntro}.\\

\noindent  When $p>0$, the situation is  more involved. The original computations of \cite{BTM} fail due to inseparability phenomena. Instead, we adjust the notion of  ``good" for the pair of points (Definition \ref{Good}) in order to   refine   the argument of Rovinsky.  In particular, we use the field-theoretic notion of ``regular" element rather than the group-theoretic notion of ``primitive"  element used in \cite{BTM}. To show that every line can be shifted to a   line whose image contains    a ``good" pair of points  (Lemma \ref{Shift}),  one can invoke Bertini theorems (\cite[Cor. 6.11.3]{Jouanolou} when $k$ is infinite and  \cite[Thm. 1.6]{CP} when $k$ is finite); we also give  an alternative, more elementary argument due to Akio Tamagawa in Remark \ref{ShiftRem}. This reduces the problem to show that $\overline{\psi}$ (or $\overline{\psi}^{-1}$) maps every line in $F^\times/k^\times$ whose image contains  a good pair of points $(\overline{x}_1,\overline{x}_2)$  isomorphically to a line in  $F^{'\times}/k^{'\times}$. Actually, we cannot prove this directly when $p>0$. The issue is that, when $p>0$, the Bogomolov-Tschinkel parametrization of such line by the set $\frak{I}(\overline{x}_1,\overline{x}_2)$ introduced in Subsection \ref{Notation}
  is much rougher than in \cite[Thm. 22]{BTM} - up to prime-to-$p$ powers and certain affine transformations with $F^p$-coefficients (Lemma \ref{ML}); this is due to the apparition of constants in $F^p$ when one integrates differentials forms. Lemma \ref{ML}   is however enough to show that there exists a unique $m\in \Z$ normalized as 
 \begin{equation}\label{mcond}
\begin{tabular}[t]{ll}
 $|m|= 1$ &if $p=0,2$\\ 
 $1\leq |m|\leq \frac{p-1}{2}$ &if $p>2$;
 \end{tabular}
 \end{equation} 
such that $\overline{\psi}^m$ induces a bijection from lines in $F^\times/p$ to lines in $F'^{\,\times}/p$ (Proposition \ref{Lines}). So that Lemma \ref{FTPG}  gives a unique field isomorphism $\phi: F\tilde{\rightarrow} F'$ such that the resulting isomorphism of  groups $ \phi:F^\times \tilde{\rightarrow} F'^{\,\times }$ coincides with  $\overline{\psi}^m$   on $F^\times/p $. This concludes the proof if $p=0$. But if $p>0$, the extension $F/F^p$ is much smaller (finite-dimensional!) and one has to perform an additional descent step (Section \ref{TheEnd}) to show that $m=\pm 1$ and $\phi$ coincides with $\overline{\psi}^{\pm 1}$ on $F^\times/k^\times$ (not only on $F^\times/p$).\\

\noindent We limited our exposition to function fields, which  are those of central interest in algebraic geometry. However, some of our results extend to more exotic fields provided they behave like function fields. For instance, Theorem \ref{TAD} works for the class of regular field extensions $F$ of transcendence degree $\geq 2$ over a perfect field $k$ such that for every subfield $k\subset E\subset F$ of transcendence degree $ 2$ over $k$, the algebraic closure of $E$ in $F$ is a finite extension of $E$. We do not elaborate on this.

\subsection{Acknowledgements} 
 The first author was partially supported by the I.U.F. and the ANR grant ANR-15-CE40-0002-01. The second author was partially
supported by NSF grant DMS-1601680 and by the Laboratory of Mirror
Symmetry NRU HSE, RF Government grant, ag.\ no.\ 14.641.31.0001. This project was initiated at the occasion of a working group on  geometric birational anabelian geometry at the Courant Institute of Mathematics, NYU.  The
authors would like to thank F. Bogomolov and Y. Tschinkel for inspiring exchanges, B. Kahn for his explanations concerning the Bass-Tate conjecture and  A. Merkurjev, M. Morrow and A. Tamagawa for helpful remarks on a first version of this text. They also thank A. Tamagawa for suggesting the more elementary proof of Lemma \ref{Shift} presented in Remark \ref{ShiftRem}.

\section{Milnor K-rings and algebraic dependence}\label{Recons}

\noindent Let   $K_*^M(-)$ denote the Milnor K-ring functor, from the groupoid $\mathcal{F}$ of fields to the groupoid $\mathcal{A}$ of associative $\Z_{\geq 0}$-graded anti-commutative rings.\\

\noindent  For   $p= 0$ or a prime, let  $\mathcal{F}_p\subset \mathcal{F}$ denote the full subgroupoid of fields of characteristic $p$.

\subsection{Some geometric observations} Let $F/k$ be a finitely generated regular field extension.  
\subsubsection{} Let $x_1,\dots, x_n\in F$ be such  that $k(x_1,\dots, x_n)$ has transcendence degree $n$ over $k$. Then

\begin{lemm}\label{GeoLem}There exists a divisorial valuation $v $ of $F$ such that 
\begin{itemize}[leftmargin=* ,parsep=0cm,itemsep=0cm,topsep=0cm]
\item $v(x_1)\not=0$, $v(x_2)=\dots=v(x_n)=0$;
\item the images $\overline{x}_2,\dots, \overline{x}_n $ of $x_2,\dots, x_n $ in the residue field $k(v)$ are algebraically independent over $k$.
\end{itemize}
\end{lemm}

\noindent In particular, one can compute explicitly the image of $\langle x_1,\dots,x_n\rangle$ \textit{via} the residue map
$\partial_v^n:K_n^M(F)\rightarrow K_{n-1}^M(k(v))$
as $$\partial_v^n(\langle x_1,\dots,x_n\rangle)=\pm v(x_1)\langle \overline{x}_2,\dots,\overline{x}_n\rangle=\pm \langle \overline{x}_2^{v(x_1)}, \overline{x}_3,\dots,\overline{x}_n\rangle.$$

\begin{proof} Fix a normal projective model $X/k$ of $F/k$. Each $x_i$ defines a dominant rational function $x_i:X\dashrightarrow \PP^1_k$ such that $$\underline{x}=(x_1,\dots,x_n):X\dashrightarrow (\PP^1_k)^n$$ is again dominant. Choose any open subscheme $U \subset X$ over which the map  $\underline{x}:U\rightarrow  (\PP^1_k)^n$, $i=1,\dots, n$ is defined. Then, up to replacing $X$ by the normalization of the Zariski closure of the graph of $\underline{x}|_U$ in $X\times (\PP^1_k)^n$, one may assume that the maps $x_i$, $i=1,\dots, n$ and $\underline{x}$ are defined over $X$, and  surjective. Choose an irreducible divisor $D\in Div(X)$ with $v_D(x_1)\not=0$. Then (since $\underline{x}(D)$ has codimension at most $1$ in $(\PP^1_k)^n$) the restriction $\underline{x}|_D:D\rightarrow  (\PP^1_k)^n$ surjects onto $0\times (\PP^1_k)^{n-1} \subset (\PP^1_k)^n$; in particular   the elements $\overline{x}_i:=x_i|_D\in k(D)$, $i=2,\dots, n$ remain algebraically independent over $k$   and $v_D(x_i)=0$,  $i=2,\dots, n$. 
\end{proof}

\subsubsection{}\label{Free} The abelian group $ F^\times/k^\times$ is a free abelian group since it embeds into a free abelian group. This  follows from the exact sequence 
$$0\rightarrow k^\times\rightarrow F^\times\rightarrow Div(X),$$
where $X/k$ is any normal projective model $X/k$ of $F/k$, and the fact that $Div(X)$ is a free abelian group.
\subsection{Type}\label{Type}We  say that a  field $F$ is a function field over a field $k$ (or with field of constants $k$) if $F$ is a finitely generated regular field extension of transcendence degree $\geq 1$ over $k$. We say that a function field $F$ over $k$ is of type 1 (resp. of type 2) if  $k$ is algebraically closed (resp. finite), in which case $k$ is uniquely determined by $F$ (see Lemma \ref{Constants} below).  \\

\noindent   Let $\widetilde{\mathcal{F}}\subset \mathcal{F}$, $\widetilde{\mathcal{F}}_p\subset \mathcal{F}_p$ denote the full subcategories of  function fields of type $1,2$. \\ 

\begin{lemm}\label{Constants}For $F\in \widetilde{\mathcal{F}}$, the type, the multiplicative group $k^\times$ of the field of constants and the characteristic $p$ of $F$  are determined by   $F^\times=K_1^M(F)$ as follows.\\
 $$
 \begin{tabular}[t]{c|c|c|c}
 Type&Torsion subgroup&$k^\times$&$p$\\
 & of $F^\times$&&\\
 \hline\\
 1&Infinite&Divisible&$0$ or unique  $p$ such that\\
 & & subgroup of $F^\times$ &$p\cdot :k^\times \rightarrow k^\times$ is an isomorphism\\
 2&Finite&Torsion&\\
 && subgroup of $F^\times$&Unique   $p$ such that $\log(|k^\times|+1)\in \Z\log(p)$
\end{tabular}$$
 
\end{lemm}

 \noindent In the following, we sometimes write $F/k\in \widetilde{\mathcal{F}}$ instead of  $F\in  \widetilde{\mathcal{F}}$ implicitly meaning that $k$ is the field of constants of $F$.\\
\subsection{Detecting algebraic dependence}\label{ACBF} For  $F\in   \mathcal{F}_p $ and a prime $\ell\not= p$, let  $D K^M_*(F) \subset K^M_*(F)$ (resp.  $TK^M_*(F) \subset K^M_*(F)$)  denote  the (two-sided) ideal of elements which are  infinitely $\ell$-divisible (resp.  torsion) with respect to the  structure of $\Z$-module
 and write $$
 \begin{tabular}[t]{ll}
 $ \overline{K}^{1,M}_*(F) :=K^M_*(F)/D K_*^M(F)$;\\
 $  \overline{K}^{2,M}_*(F) :=K^M_*(F)/TK_*^M(F)$. \\
\end{tabular}$$ 
Then  $K^M_*\rightarrow   \overline{K}^{i,M}_*$ is a morphism of functors from $\mathcal{F}$ to $\mathcal{A}$, $i=1,2$.\\

\begin{theo}\label{Div}
For $F \in \widetilde{\mathcal{F}}_p $ of type $i$ and every $x_1,\dots, x_n\in F^\times $ consider the following assertions.
\begin{itemize} 
\item [(i)]  $\langle x_1,\dots,x_n\rangle=0$ in $\overline{K}_{n}^{i,M}(F) $;
 
\item [(ii)] the transcendence degree of $k(x_1,\dots, x_n)$ over $k$ is  $\leq n-1$.
\end{itemize}
Then,  (i) $\Rightarrow$ (ii) and  (ii) $\Rightarrow$ (i) if $i=1$ or if $i=2$  and $n\leq 2$.
 
\end{theo}
\begin{proof} The assertion for $n=1$ follows from \ref{Free}.\\

\noindent   (i) $\Rightarrow$ (ii): We proceed by induction on $n$.  Assume  $k(x_1,\dots, x_n)$ has transcendence degree $n$ over $k$ and choose a place $v$ of $F$ as in   Lemma \ref{GeoLem}. By induction hypothesis
 $\partial_{v }^n(\langle x_1,\dots,x_n\rangle)\not= 0$ in $\overline{K}_{ n-1}^{i,M}(k(v))$ hence, \textit{a fortiori}, $\langle x_1,\dots,x_n\rangle\not=0$ in $\overline{K}_{ n}^{i,M}(F) $. \\

\noindent  (ii) $\Rightarrow$ (i): Assume $E=k(x_1,\dots, x_n) $ has transcendence degree~$d\leq n-1$ over $k$. Since $\langle x_1,\dots,x_n\rangle$ lies in the image of   the restriction map $\overline{K}_n^{i,M}(E) \rightarrow \overline{K}_n^{i,M}(F) $, it is enough to show that $\overline{K}_n^{i,M}(E)=0$. If $F$ is of type 1, this follows from Tsen's theorem \cite{Lang-Tsen}, which ensures that  
 $H^n(E, \Z_\ell(n))=0$ and  from the Bloch-Kato conjecture \cite{Voe, Voe1}.  If $F$ is of type 2, this follows from  the $n=2$ case of the Bass-Tate conjecture \cite[Thm. 1]{Tate}.
\end{proof}

\noindent Let $F,F'\in\widetilde{\mathcal{F}}_p$ of the same type $i$. For a  morphism  $\overline{\psi}_*:\overline{K}^{i,M}_*(F)\rightarrow \overline{K}^{i,M}_*(F')$   of $\Z_{\geq 0}$-graded rings and integer $n\geq 1$, consider the assertion

\begin{center}
$(\circ,n)$ For every $x_1,\dots, x_n\in F$, $\overline{\psi}_1(\overline{k(x_1,\dots, x_n)^F}^\times/k^\times)\,\circ\, \overline{k'(\psi_1(x_1),\dots, \psi_1(x_n))^{F'}}/k^{'\times}$ \\
\end{center}

\noindent where $\psi_1:F^\times\rightarrow F^{'\times}$ denotes any set-theoretic lift of $\overline{\psi}_1$ and $\circ$ means the symbol $=$ or $\subset$.

 \begin{coro}\label{DivCoro} Then Assertion  $(\subset,n)$ holds for every $n$ if $i=1$ and for $n\leq 2$ if $i=2$. If  $\overline{\psi}_*$ is an isomorphism  of $\Z_{\geq 0}$-graded rings, Assertion $(=,n)$ holds for every $n$ if $i=1$ and for $n\leq 2$ if $i=2$.
 \end{coro}
 \begin{rema}\textit{}
 \begin{itemize}[leftmargin=* ,parsep=0cm,itemsep=0cm,topsep=0cm]
 \item When $F$ is of type 2 the implication (ii) $\Rightarrow$ (i) of Theorem \ref{Div} and the assertions $(\subset,n)$, $(=,n)$ of Corollary \ref{DivCoro} for every $n$ are predicted by the Bass-Tate conjecture in positive characteristic \cite[Question, p.390]{BassTate}. See also \cite{Kahn} for the relation between the Bass-Tate conjecture in positive characteristic and classical motivic conjectures. 
 \item For our applications, we only need the $n\leq 2$ cases of Corollary \ref{DivCoro}. In particular, for function fields of type 1, we only need the $n=2$ case of the Bloch-Kato conjecture, an earlier theorem of Merkurjev and Suslin \cite{MS}.
  \end{itemize}
\end{rema}
\subsection{Reconstructing function fields} The $n\leq 2$ case of Corollary \ref{DivCoro} implies that for $i=1,2$ and  for every $F,F'\in\widetilde{\mathcal{F}}_p$ of type $i$ and isomorphism  $\overline{\psi}_*:\overline{K}^{i,M}_*(F)\rightarrow \overline{K}^{i,M}_*(F')$ 
 of $\Z_{\geq 0}$-graded rings the induced isomorphism of multiplicative groups  $\overline{\psi}_1$ preserves algebraic dependence. This implies the following. For $i=1,2$, let  
$$
\hbox{\rm Isom}(\overline{K}_*^{i,M}(F),\overline{K}_*^M(F'))  
$$ 

denote the set of isomorphisms of $\Z$-graded rings
$ 
 \overline{K}_*^{i,M}(F)\tilde{\rightarrow} \overline{K}_*^{i,M}(F') 
 $
 and $$
\overline{\hbox{\rm Isom}}(\overline{K}_*^{i,M}(F),\overline{K}_*^{i,M}(F')) 
$$ 
their  quotients by the natural action of $\Z/2$. 

\begin{coro}\label{FullMT}\textit{}Let $F,F' \in\mathcal{\tilde{F}}$.  Then
\begin{itemize}
\item[(i)]  $  \hbox{\rm Isom}(K_*^M(F),K_*^M(F'))\not=\emptyset $ only if $F,F' $ are of the same type $i$ and characteristic $p$. 
\item[(ii)]    If  $F,F'\in\mathcal{\tilde{F}}_p$ are   of the same type $i$ and have transcendence degree $\geq 2$ over their fields of constants   then  
  the natural functorial map $$\hbox{\rm Isom}(F,F')\rightarrow \overline{\hbox{\rm Isom}}(\overline{K}_*^{i,M}(F),\overline{K}_*^{i,M}(F')) $$
is bijective.
 \end{itemize}
\end{coro}

\begin{proof}  Part (i) follows from Lemma \ref{Constants}. For Part (ii), 
one has a canonical  commutative diagram 
$$  \xymatrix{&\overline{\hbox{\rm Isom}}^\equiv(F^\times/k^\times,F^{'\times}/k^{'\times})\ar@{^{(}->}[dr]&\\
\hbox{\rm Isom}(F,F')\ar[dr]\ar[ur]^{(2)}&& \overline{\hbox{\rm Isom}}(F^\times/k^\times,F^{'\times}/k^{'\times})\\
&\overline{\hbox{\rm Isom}}(\overline{K}_*^{i,M}(F),\overline{K}_*^{i,M}(F'))\ar@{^{(}->}[ur]_{\overline{K}_1^{i,M}}\ar[uu]^{(1)}&}$$
where the factorization (1) follows from Corollary \ref{DivCoro}. So the conclusion follows from the fact that the map (2) is bijective by Theorem \ref{TAD}.
\end{proof}
 
 \begin{coro}\label{Auto}\textit{}Let $F  \in\mathcal{\tilde{F}}$ of type $i$.  Then one has the following group isomorphisms
 $$\hbox{\rm Aut}(F)\tilde{\rightarrow}\overline{\hbox{\rm Aut}}^\equiv(F^\times/k^\times)\tilde{\rightarrow}\overline{\hbox{\rm Aut}}(\overline{K}_*^{i,M}(F))$$
 
\end{coro}

  \begin{rema} For convenience, we stated Corollary \ref{FullMT} with the whole rings $ \overline{K}_*^{i,M}(F)$ but,  it is actually enough to consider the datum of $ \overline{K}_{\leq 2}^{i,M}(F)$ that is   the canonical pairing
  $$\langle -,-\rangle: \overline{K}_1^{i,M}(F)\otimes \overline{K}_1^{i,M}(F) \rightarrow \overline{K}_2^{i,M}(F),$$
(for $i=1$, one can even replace $\overline{K}_2^{1,M}(F)$  with its $\ell$-adic completion $ H^2(F,\Z_\ell(2))$). 
\end{rema}
 
  \begin{rema}From Lemma \ref{Constants}, one can reconstruct the field of constants  $k$ of a function field $F$ in $ \mathcal{\tilde{F}}$ from $K_1^M(F)$. In general, one may ask for a relative version of Corollary \ref{FullMT} that is replacing the functor $K_*^M(-)$ with  the functor   sending a finitely generated field extension $F$ of a perfect field $k$ to the morphism $K_*^M(k)\rightarrow K_*^M(F)$.
\end{rema}
\noindent The remaining part of the paper is devoted to the proof of Theorem \ref{TAD}.

\

\section{Proof of Theorem \ref{TAD}}\label{RPL1}
\noindent We refer to Subsection \ref{StrategyIntro} for a description of the strategy of the proof.  

\subsection{Notation}\label{Notation}\textit{}  \\
  
\noindent Let $k$ be a perfect field of characteristic $p\geq 0$ and let $F/k$ be a finitely generated regular  field extension of transcendence degree  $\geq 2$. \\

 \noindent For   every  subfield $k\subset E \subset  F$ and $\overline{x}\not= \overline{y} \in F^\times/E^\times$, write\\
  $$\frak{l}_{E}(\overline{x},\overline{y})=(Ex \oplus Ey)^\times/E^{\times}\subset F^\times/E^{\times}$$
  for  the corresponding  line in $F^\times/E^{\times}$.
 
 \begin{defi} We say that $\overline{x},\overline{y}\in F^\times/k^\times$ are {\it $p$-multiplicatively dependent} and write   $\overline{x}\sim_p  \overline{y}$ if  $$
 \overline{x}^{\Z}\cap \overline{y}^{\Z}\not=1$$
 in $F^{ \times}/p$.
 The relation $\sim_p$ is an equivalence relation on $F^\times/k^\times$.
\end{defi}
 
\
  
 \noindent For $\overline{x}_1,\overline{x}_2,\overline{y}_1,\overline{y}_2\in F^\times/k^\times$ and some (equivalently, every) lifts $x_1,x_2, y_1,y_2\in F^\times$, write
 $$\mathcal{I}(\overline{x}_1,\overline{x}_2,\overline{y}_1,\overline{y}_2):= (\overline{k(x_1/x_2)^F}^\times\cdot  x_2)\bigcap (\overline{k(y_1/y_2)^F}^\times\cdot y_2)$$
 and  
 
$$  \frak{I}(\overline{x}_1,\overline{x}_2)=\bigcup_{y_1,y_2}  \mathcal{I}(\overline{x}_1,\overline{x}_2,\overline{y}_1,\overline{y}_2)
$$
where the union is over all $y_i\in \overline{k(x_i)^F}^\times$, $x_i\not\sim_p y_i $, $i=1,2$. \\

    \

\noindent Let $\overline{\mathcal{I}}(\overline{x}_1,\overline{x}_2,\overline{y}_1,\overline{y}_2)$,    $ \overline{\frak{I}}(\overline{x}_1,\overline{x}_2)$ denote the images of $ \mathcal{I}(\overline{x}_1,\overline{x}_2,\overline{y}_1,\overline{y}_2)$,  $ \frak{I}(\overline{x}_1,\overline{x}_2)$ in $F^\times/k^\times$ respectively.\\
 
\noindent For every $\overline{I}\in \overline{\frak{I}}(\overline{x}_1,\overline{x}_2)$, let  $\frak{l}^\circ(\overline{I},\overline{x}_1)$  denote the   set of all   $y_1\in \overline{k(x_1)^F}$, $x_1\not\sim_py_1$   such that for some (equivalently, every) lift $I\in F^\times$ of $\overline{I}\in F^\times/k^\times$, one has $I\in \overline{k(y_1/y_2)^{F}}^\times\cdot y_2$ for some $y_2\in \overline{k(x_2)^F}$, $x_2\not\sim_py_2$. \\

\noindent Similarly, let    $\frak{l}^\circ(\overline{I},\overline{x}_2)$ denote the  set of all  $y_2\in \overline{k(x_2)^F}$,   $x_2\not\sim_py_2$  such that for some (equivalently, every) lift $I\in F^\times$ of $\overline{I}\in F^\times/k^\times$, one has $I\in \overline{k(y_1/y_2)^{F}}^\times\cdot y_2$  for some $y_1\in \overline{k(x_1)^F}$, $x_1\not\sim_py_1$.\\

\noindent Set also $$\frak{l}(\overline{I},\overline{x}_i)=\frak{l}^\circ(\overline{I},\overline{x}_i)\cup \lbrace x_i\rbrace,\; i=1,2.$$

\noindent Let $\overline{\frak{l}}^\circ(\overline{I},\overline{x}_i)$, $\overline{\frak{l}}(\overline{I},\overline{x}_i)$ denote the images of $\frak{l}^\circ(\overline{I},\overline{x}_i)$, $\frak{l}(\overline{I},\overline{x}_i)$ in $F^\times/k^\times$ respectively.\\

  \subsection{Recollection on  differentials}

\begin{lemm}
\label{DiffTr}
For $x_1,\dots, x_n\in F$, the following are equivalent
\begin{itemize}[leftmargin=0.7cm ,parsep=0cm,itemsep=0cm,topsep=0cm]
\item $x_1,\dots, x_n\in F$ is a separating  transcendence basis for $F/k$, i.e.  $F/k(x_1,\dots, x_n)$ is a finite separable field extension;
\item $dx_1,\dots, dx_n\in \Omega^1_{F|k}$ is an $F$-basis of $\Omega^1_{F|k}$.
\end{itemize}
If $p>0$, these are also equivalent to
\begin{itemize}[leftmargin=0.7cm ,parsep=0cm,itemsep=0cm,topsep=0cm]
\item $x_1,\dots, x_n\in F$ is a  $p$-basis of $F/k$ that is a $\mathbb{F}_p$-basis of $F/F^p$.\\
\end{itemize}
\end{lemm}

 \noindent See \textit{e.g} \cite[\S 27, Thm 59; \S 38, Thm. 86]{Mats} for the proof.  Recall that since $k$ is a perfect field, the extension $F/k$ always admits a separating  transcendence basis, and that if $p=0$, every  transcendence basis is separating.\\

\begin{coro}\label{kerd}
$\ker(d:F\rightarrow \Omega_{F|k}^1)= F^p.$
\end{coro}

\noindent In particular:
\subsubsection{}\label{Diff1}For every $x\in F$, the following are equivalent:
\begin{itemize} 
\item [(i)]  $x\in F\setminus F^p$;
\item [(ii)]  $dx\not= 0$;
\item [(iii)]  $x$ is a separating transcendence basis for $\overline{k(x)^F}/k$.
\end{itemize} 
If $x\in F$ verifies the above equivalent conditions (i), (ii), (iii),  for every $0\not=f\in \overline{k(x)^F}$  there exists a unique $f':=f'(x)\in  \overline{k(x)^F}$ such that $df=f'dx$.

\subsubsection{}\label{Diff2}For every $x_i\in F$, $i=1,2$, with $dx_1, dx_2\in \Omega_{F|k}^1$  linearly independent over $F$,  and $z_i\in \overline{k(x_i)^F}\setminus k^\times$, $i=1,2$, the following are equivalent:
\begin{itemize} 
\item [(i)]  $z_1/z_2\in F \setminus F^p$;
\item [(ii)]  $dz_1, dz_2$ are  linearly independent over $F$;
\item [(iii)]  $z_i'\not= 0$, $i=1,2$;
\item [(iv)]  $z_i\in   F \setminus F^p$, $i=1,2$.
\end{itemize} 

\

 \subsection{Shifting multiplicatively lines to lines passing through good pairs}\textit{}  \\
  
\noindent Let $k$ be a perfect field of characteristic $p\geq 0$ and let $F/k$ be a finitely generated regular  field extension of transcendence degree  $\geq 2$.

 \begin{defi} We say that $x\in F$ is {\it $F/k$-regular}  if $x$ is transcendental over $k$ and   $F/k(x)$ is a regular field extension.
\end{defi}

  \begin{defi}\label{Good} We say that $(\overline{x}_1,\overline{x}_2)\in F^\times/k^\times$ is a {\it good pair}   if  for some (equivalently, every) lifts $x_1,x_2\in F^\times $ of $\overline{x}_1,\overline{x}_2\in F^\times/k^\times$, $x_1$ is $F/k$-regular and 
 $dx_1, dx_2\in \Omega_{F|k}^1$ are  linearly independent over $F$.
\end{defi}

\begin{lemm}\label{Shift}For every  $x\in F^\times\setminus F^{\times p}$,     there exists a $F/k$-regular   $y \in F^\times$  such that  $dx,dy$ are linearly independent over $F$.
\end{lemm}
\noindent In particular, for every  $\overline{x}\in F^\times/k^\times$, $\overline{x}\not=1$ in $F^\times/p$   there exists  $\overline{y}\in F^\times/k^\times$ such that  $(\overline{y},\overline{x}\overline{y})\in F^\times/k^\times$ is a good pair.

  \begin{proof} This follows from Bertini theorems.  Since $F$ has finite transcendence degree   $r\geq 2$ over $k$ and $k$ is perfect, there exist $x_2,\dots ,x_r\in F^\times$ such that $dx,dx_2,\dots,dx_r$   are linearly independent over $F$. 
 Write $E:=k(x_2,\dots, x_r)$. Let $z\in F^\times$ with minimal polynomial 
  $$P_z=T^d+\sum_{0\leq i\leq d-1}a_i T^i\in E(x)[T]$$
  over   $E(x)$. Then   $dx, dz$ are linearly dependent over $F$ if and only if 
 $\partial a_i/\partial x_j=0 $, $i=0, \ldots, d-1,\, j=2, \ldots n$, or equivalently,
   $P_z\in E^p(x)[T]$. In particular, any $y\in E(x)\setminus E^p(x)$ will have the property that $dx,dy$  are linearly independent over $F$.  We claim that one can find $y\in E(x)\setminus E^p(x)$ such that $y$ is $F/k$-regular. Fix a normal quasi-projective model $X/k$ of $F/k$ such that $x,x_2,\dots, x_r$ induce a finite dominant morphism $\underline{x}:X\rightarrow \mathbb P^r_k$. 
It is enough to show  there exists a  homogeneous polynomial $P\in k[x,x_1,\dots, x_r]\setminus k[x,x_1^p,\dots, x_r^p]$ such that the fiber at $0$ of the composite map $$y:X\stackrel{\underline{x}}{\rightarrow }\PP^r_k\stackrel{P}{\rightarrow}\PP^1_k$$ is geometrically irreducible. Viewing $y$ as an element in $E(x)$, we deduce that  $y$ is $F/k$-regular, by \cite[9.7]{EGA4}. If $k$ is infinite, this follows directly from \cite[Cor. 6.11.3]{Jouanolou}. If $k$ is finite, this follows from \cite[Thm. 1.6]{CP} and Lemma \ref{Density} below (note that $ \frac{1}{p^{r-1}}<1$ since $r\geq 2$). \end{proof}

\noindent Set  $S':=k[x,x_2^p,\dots, x_r^p ]\subset k[x,x_2,\dots, x_r ]$ and define the density of $S'$ as 
$$\delta(S')=\lim_{d\rightarrow +\infty} \frac{|S'\cap S_d|}{|S_d|},$$  
where $S_d\subset k[x,x_2,\dots, x_r]$ denotes the set of homogeneous polynomials of degree $d$. Then

\begin{lemm}\label{Density}  $\delta(S')=\frac{1}{p^{r-1}}.$
\end{lemm}

\begin{proof} One has
 $|S_d|=\binom{d+r-1}{d}$ and $|S'\cap S_d|=\sum_{0\leq n\leq \lfloor\frac{d}{p}\rfloor}\binom{n+r-2}{n}$.
To estimate $\frac{|S'\cap S_d|}{|S_d|}$, write $a:= \lfloor\frac{d}{p} \rfloor$, $b:=d-pa$ and $$Q_r(T):=\prod_{1\leq k\leq r-2}(T+k) -T^{r-2}=\sum_{0\leq i\leq r-3}a_iT^i\in \Z[T].$$
Then $$\frac{|S'\cap S_d|}{|S_d|}=\binom{d+r-1}{d}^{-1}\sum_{0\leq n\leq a}\binom{n+r-2}{n}=\frac{r-1}{\prod_{1\leq k\leq r-1}(1+\frac{b+k}{pa})}\frac{1}{(pa)^{r-1}}\sum_{0\leq n\leq a}\prod_{1\leq k\leq r-2}(n+k)$$
 with
 $$\sum_{0\leq n\leq a}\prod_{1\leq k\leq r-2}(n+k)=\sum_{0\leq n\leq a} n^{r-2}+\sum_{0\leq i\leq r-3}a_i\sum_{0\leq n\leq a}n^i.$$
 By Faulhaber's formula
 $$\sum_{1\leq n\leq a}n^i=\frac{a^{i+1}}{i+1}+\frac{1}{2}a^i+\sum_{2\leq k\leq i}\frac{B_k}{k!}\frac{i!}{(i-k+1)!}a^{i-k+1}$$
 (where the $B_k$ are the Bernoulli numbers) one gets $$\sum_{0\leq n\leq a}\prod_{1\leq k\leq r-2}(n+k)\sim \frac{a^{r-1}}{r-1}$$
 whence the assertion.   \end{proof}
   
   \begin{rema}\label{ShiftRem}We resorted to Bertini theorems, which provide a  conceptually natural proof of Lemma \ref{Shift} but one can give more elementary arguments. If $k$ is infinite,  the Galois-theoretic Lemma  \cite[Chap. VIII, Lem. in Proof of Thm. 7]{LangAG}  already shows there exists (infinitely many)  $0\not= a\in k$ such that   $y:=ax_2+x$ is $F/k$-regular; by construction $dx, dy$ are   linearly independent over $F$. If $k$ is finite,  Akio Tamagawa suggested the following arguments. Fix a smooth (not necessarily proper) model $X/k$ of $F/k$ and a closed point $t\in X$. Let $\widetilde{X}\rightarrow X$ denote the blow-up of $X$ at $t$ and $D_t=\P^{n-1}_{k(t)}\subset \widetilde{X}$ the exceptional divisor. Fix a non-empty affine subset $U=\hbox{\rm Spec}(R)\subset \widetilde{X}$ such that $Z:=U\cap D_t$ is non-empty and the rational map $x:\widetilde{X}\rightarrow X\dashrightarrow \mathbb{A}^1_k$ given by $x$ is defined over $U$. We endow $Z$ with its reduced subscheme structure. Since $Z$ is irreducible, one can write $Z=\hbox{\rm Spec}(R/P)$ for some prime ideal $P$ in $R$. Pick $y\in R\setminus R\cap F^p\overline{k(x)^F}$ such that $f\hbox{\rm mod} P\in R/P\setminus k(t)$. Let $\overline{k[y]^F}\subset F$ denote the normal closure of $k[y]$ in $\overline{k(y)^F}/k(y)$. Since $R$ is smooth over $k$ hence normal, $\overline{k[y]^F}\subset R$. Also, by our choice of $y$, the morphism $\overline{k[y]^F}\hookrightarrow R\twoheadrightarrow R/I$ is injective whence  the fraction field $ \overline{k(y)^F}$ of $\overline{k[y]^F}$, embeds into the fraction field $k(t)(x_2,\dots, x_n)$ of $R/I$.  Since $\overline{k(y)^F}/k$ is regular, $k(t)\otimes_k \overline{k(y)^F}\simeq  k(t)\cdot\overline{k(y)^F}\subset k(t)(x_2,\dots, x_n)$. From Lur\"{o}th theorem, one thus has $k(t)\otimes_k \overline{k(y)^F}=k(t)(T)$ for some   $T\in k(t)(x_2,\dots, x_n)$ transcendent over $k(t)$. In other words, $\overline{k(y)^F}$ is the function field of a $k(t)/k$ form $C_y$ of $\P^1_k$. Since $k$ is finite hence perfect with trivial Brauer group, $C_y\simeq \P^1_k$. This shows $y\in F$ is $F/k$ regular. Since $y\notin F^p\overline{k(x)^F}$, $dx$, $dy$ are linearly independent over $F$.
   \end{rema}
  \subsection{Approximating lines passing through good pairs up to powers}\textit{}  \\
  
\noindent  Let $k$ be a perfect field of characteristic $p\geq 0$ and let $F/k$ be a finitely generated regular  field extension of transcendence degree  $\geq 2$.

\begin{lemm}\label{ML}Let $(\overline{x}_1,\overline{x_2})\in F^\times/k^\times$ be a good pair. Then, for every   $\overline{I}\in \overline{\frak{I}}(\overline{x}_1,\overline{x}_2)$ there exists $m\in \Z$ satisfying  (\ref{mcond}), $N\in \Z$, $\alpha=\alpha(\frac{x_1}{x_2})\in  \overline{k(\frac{x_1}{x_2})^F}^{\times }$ and $c\in k^\times$ such that $$I^m=\alpha(\frac{x_1}{x_2})^p(x_1^m-c\frac{  x_1^{pN}}{ x_2^{pN}}x_2^m)\in \frak{l}_{F^p}(x_1^m,x_2^m).$$
Furthermore,  for every $y_i\in \overline{k(x_i)^F}$, $y_i\not\sim_p x_i$, $i=1,2$ such that $I\in \mathcal{I}(\overline{x}_1,\overline{x}_2,\overline{y}_1,\overline{y}_2)$, one has for $i=1,2$
$$y_i^m=\alpha_i^p(x_i^m-c_ix_i^{pN})$$
  for some $\alpha_i\in \overline{k(x_i)^F}$, $c_i\in k^\times$ with the condition $c=c_1/c_2$.
\end{lemm}

 \noindent \textit{Proof.} The proof when $p=0$ is significantly simpler since $\ker(d)=k$ (recall $F/k$ is regular). We carry out the proof for $p>0$ and just mention the simplifications that occur for $p=0$. The results for $p=0$ are exactly similar but with elements in $F^p$ replaced by elements in $k$. \\

\noindent We are to determine the possible    $y_i\in \overline{k(x_i)^F}^\times$, $y_i\not\sim_p x_i$, $i=1,2$ such that $  \mathcal{I}(\overline{x}_1,\overline{x}_2,\overline{y}_1,\overline{y}_2)\not=\emptyset$ and for all such $y_i$, $i=1,2$ the elements $I\in  \mathcal{I}(\overline{x}_1,\overline{x}_2,\overline{y}_1,\overline{y}_2)$. So assume $  \mathcal{I}(\overline{x}_1,\overline{x}_2,\overline{y}_1,\overline{y}_2)\not=\emptyset$ and fix $I\in   \mathcal{I}(\overline{x}_1,\overline{x}_2,\overline{y}_1,\overline{y}_2)$.\\
 
  \noindent   For $z=x,y$ we have: 
 $I/z_2\in \overline{k(z_1/z_2)^F}^\times,$
 so that there exists $A_z \in  \overline{k(z_1/z_2)^F}^\times$ with $$\frac{d(I/z_2)}{I/z_2}=A_z\frac{d(z_1/z_2)}{z_1/z_2}$$
 (here, we use $z_1/z_2\notin F^p$).  Equivalently,  $$\frac{dI}{I}=\frac{dz_2}{z_2}(1-A_z)+\frac{dz_1}{z_1}A_z. $$
 
 \noindent We deduce
 
$$ A_x\frac{dx_1}{x_1}-A_y\frac{dy_1}{y_1}=(A_x-1)\frac{dx_2}{x_2}-(A_y-1)\frac{dy_2}{y_2}\in Fdx_1\cap Fdx_2=0.$$

\noindent Whence, setting $dy_i=y'_idx_i$ with $y_i' \in \overline{k(x_i)^F}$ (here, we use $x_i\notin F^p$) and using that $dx_1,dx_2 $ are linearly independent over $F$, we obtain
\begin{equation}
\label{functions1}
A_x= A_y\frac{x_1y'_1}{y_1}=(A_y-1)\frac{x_2y_2'}{y_2}+1.
\end{equation}
Set 
\begin{equation}
\label{functions2}
f_i:=\frac{x_iy'_i }{y_i}\in\overline{k(x_i)^F}^\times,\; i=1,2.
\end{equation}
 We obtain
\begin{equation}
\label{functions22} A_x=A_y f_1,\;\; 
 A_y(f_1-f_2)=1-f_2 \end{equation}
By Lemma \ref{Tech} (i)  below and the fact that $y_2\not\sim_px_2$,   $1-f_2\not=0$ hence $f_1-f_2, A_y\not= 0$. As a result  the second equation in (\ref{functions22}) can be rewritten $$A_y=\frac{1-f_2}{f_1-f_2}.$$

\noindent  So, setting $df_i=f'_idx_i$ with $f_i' \in \overline{k(x_i)^F}$, $i=1,2$ we get
 
$$
 \frac{dA_y}{A_y}=\frac{f_1'}{f_2-f_1}dx_1+\frac{(1-f_1)f_2' }{(1-f_2)(f_1-f_2)}dx_2.
$$
We also have $A_y\in  \overline{k(y_1/y_2)^F}^\times$  so that  there exists $\alpha\in \overline{k(y_1/y_2)^F}$ with 
$$\frac{dA_y}{A_y}=\alpha(\frac{y_1'}{y_1}dx_1-\frac{y'_2 }{y_2}dx_2).$$  
(here we use $y_1/y_2\notin F^p$). Since $dx_1,dx_2$ are linearly independent over $F$, one gets $$ \alpha=\frac{y_1f_1' }{y_1'(f_2-f_1)}=\frac{y_2(1-f_1)f_2' }{y_2' (1-f_2)(f_2-f_1)}$$
whence 
 $$ c:=\frac{y_1f_1'}{y_1' (1-f_1)}=\frac{y_2 f_2'}{y_2' (1-f_2)}\in \overline{k(x_1)^F}\cap \overline{k(x_2)^F}=k.$$
By Lemma \ref{Tech} (ii) below and   the fact that $y_1\not\sim_px_1$, we have   $f_1\notin F^p$ hence $c\not= 0$. Recalling that $f_i:=\frac{x_iy'_i }{y_i}$, $i=1,2$ we eventually get 
\begin{equation}\label{fequation} 
x_if_i' =cf_i(1-f_i).
\end{equation}
Now, the last steps of the proof are as follows. \\

\begin{itemize}[leftmargin=0.7cm ,parsep=0cm,itemsep=0cm,topsep=0cm] 
\item Step 1.  Considering (\ref{fequation}) for $i=1$ and using that $x_1\in F^\times$ is $F/k$-regular, we   show that the parameter $c$ necessarily lies in $ \Z$ and can be taken satisfying (\ref{mcond}); we denote it $c:=m$. Once this is settled, one can easily solve (\ref{fequation}) and determine $y_i^m$, $i=1,2$.
\item Step 2. Using the relations between $I, A_x, A_y, f_1, f_2$, we show that $$I^{m}=\alpha^p(x_1^{m}-\frac{\beta_1^p}{\beta_2^p}x_2^{m})$$ for some $\alpha\in  \overline{k(\frac{x_1}{x_2})^F}$, $\beta_i\in \overline{k(x_i)^F}$, $i=1,2$.
\item Step 3. Using that $I/x_2\in \overline{k(\frac{x_1}{x_2})^F}$, we show that there exists $N\in \Z$ such that $\beta_i=c_ix_i^N$ for some $c_i\in k^\times$, $i=1,2$.
 
\end{itemize}

  \subsubsection{Step 1} Write $$f_1=\frac{a_1A_1^p}{b_1B_1^p }$$ with $A_1,B_1, a_1,b_1\in k[x_1]$, $a_1A_1, b_1B_1\in k[x_1]$ coprime, $a_1, b_1, B_1\in k[x_1]$ monic and $a_1,b_1\in k[x_1]$ with zeros of multiplicities at most $p-1$. Then $x_1f_1' =cf_1(1-f_1)$ can be rewritten as
 $$x_1(a_1'b_1-a_1b_1')B_1^p=ca_1(B_1^pb_1-A_1^pa_1).$$
 By considering the multiplicity of a non-zero root of $a_1 $, we see that $a_1 =x_1^m$   for some $0\not= m\in \Z$ with $\frac{1-p}{2}\leq m\leq \frac{p-1}{2}$   so that, for $b_1$, we obtain
\begin{equation}\label{cequation}
((m-c)b_1- x_1b_1')B_1^p=-cx_1^mA_1^p.
\end{equation}
Since $B_1^p$ and $a_1A_1^p=x_1^mA_1^p$ are corpime and $B_1$ is monic, $B_1=1$ and (\ref{cequation}) becomes
\begin{equation}\label{ccequation}
(m-c)b_1-  x_1b_1'=-cx_1^mA_1^p.
\end{equation}

\begin{itemize}[leftmargin=* ,parsep=0cm,itemsep=0cm,topsep=0cm]
 \item If $m\not= 0$, evaluating at $0$ and using that $b_1(0)\not= 0$ (since $a_1=x_1^m$ and $b_1$ are coprime), one gets $c=m$;
 \item If $m=0$, differentiating (\ref{ccequation}), one gets $$(c+1)b_1'-+x_1b_1''=0$$
 and writing $b_1=\sum_j b_{1,j}x_1^j$ this yields 
 $$\sum_j j(c+j)b_{1,j}x_1^{j-1}=0.$$
 Again, since $b_1\notin k[x_1^p]$ by assumption, there exists $p\not| j$ such that $b_{1,j}\not= 0$, which forces $c=-j$.\\
 \end{itemize}

\noindent In any case, one may now write $c=m$ for some $0\not= m\in \Z$ with $\frac{1-p}{2}\leq m\leq \frac{p-1}{2}$ hence 
for $i=1,2$,  (\ref{fequation})  can   be rewritten   $$(\frac{(1-f_i)x_i^m}{f_i})'=0$$
 hence $$\frac{(1-f_i)x_i^m}{f_i}=\beta_i^p$$ for some
  $\beta_i\in \overline{k(x_i)^F} $ or, equivalently, $$\frac{x_i y_i'}{y_i}=f_i=\frac{x_i^m}{x_i^m+\beta_i^p}.$$
  
\noindent Whence
 
 $$\frac{(x_i^m+\beta_i^p)'}{x_i^m+\beta_i^p}=m\frac{x_i^{m -1}}{x_i^m+\beta_i^p}=\frac{ (y_i^m)'}{y_i^m}$$
and
$$d(\frac{y_i^m}{  (x_i^m+\beta_i^p)})=0$$
that is 
\begin{equation}\label{y2}
y_i^m=\alpha_i^p(x_i^m+\beta_i^p)
\end{equation}
for some $\alpha_i\in \overline{k(x_i)^F}$.\\

\begin{rema}If $p=0$, we obtain  that there exists $0\not= m\in \Z$ such that $y_1^m=\alpha_1(x_1^m+\beta_1)$ for some $\alpha_1,\beta_1\in k^\times$. Then the factoriality of $k[x_1]$ and the fact that $x_1\not\sim_0 y_1$ yields $m=\pm 1$. If $p>0$, we are not able to ensure $m=\pm 1$.
\end{rema}
  \subsubsection{Step 2}  Then,
$$A_y=\frac{1-f_2}{f_1-f_2}=\frac{\beta_2^p(x_1^m+\beta_1^p)}{\beta_2^px_1^m -\beta_1^px_2^m }$$
and $$A_x=A_yf_1=(1-\frac{\beta_1^p}{\beta_2^p}(\frac{x_1}{x_2})^{-m })^{-1}.$$ So, writing $\beta:=\frac{\beta_1}{\beta_2}$, $x:= \frac{x_1}{x_2} $ and $J=J(x):=\frac{I}{x_2}$ one gets 
$$\frac{(J^{m })'}{J^{m }}= \frac{m x^{m -1} }{x^{m }-\beta^p}=\frac{(x^{m }-\beta^p)'}{x^{m }-\beta^p}$$
Whence $$J^{m }=\alpha^p (x^{m }-\beta^p)$$ for some $\alpha\in \overline{k(\frac{x_1}{x_2})^F}$
and $$I^{m }=\alpha^p (x_1^{m }-\beta^p x_2^{m }) =\alpha^p (x_1^{m }-\frac{\beta_1^p}{\beta_2^p} x_2^{m }) $$
with $\alpha\in \overline{k(\frac{x_1}{x_2})^F}$ and   $\beta_i\in \overline{k(x_i)^F} $, $i=1,2$. \\
  \subsubsection{Step 3}  The assumption $I/x_2~\in~\overline{k(\frac{x_1}{x_2})^F}$ forces $\beta=\frac{\beta_1}{\beta_2}\in \overline{k(\frac{x_1}{x_2})^F}^\times$. This in turn imposes 

\begin{lemm}\label{Explicit}
 $\beta_1\in k^\times x_1^N,\beta_2\in k^\times x_2^N$ for some $N\in \Z$.
 \end{lemm}
\begin{proof}  Again, write $x:=\frac{x_1}{x_2}$. Up to replacing $\beta$ with $\beta w^N$ for some $N\in \Z$, one may assume $0$ is neither a zero nor a pole of $\beta$, as a function in $x$. We are going to show that, necessarily, $\beta_1,\beta_2\in k^\times$. Let $$P_\epsilon(x,T)=T^d+\sum_{0\leq i\leq d-1}a_{\epsilon, i}(x)T^i\in k(x)[T]$$
be the monic minimal polynomial of $\beta^{\epsilon}$ over $k(x)$ for $\epsilon=\pm 1$. If $\beta_1\notin k^\times  $, then $\beta_1$ or $\beta_1^{-1}$ admits at least one zero $\lambda\not= 0$.   So the relations
 $$ \beta_2^{ -d}+\sum_{0\leq i\leq d-1}a_{1, i}(x)(\beta_1^{-1})^{d-  i}\beta_2^{- i}=0$$
$$ \beta_2^{d}+\sum_{0\leq i\leq d-1}a_{-1, i}(x)\beta_1^{d- i}\beta_2^{i}=0$$
yield $\beta_2^{ -d}=0$ or $\beta_2^{ d}=0$: a contradiction. This shows that $\beta_1\in k^\times  $. Then, as $x_2$ and $w$ are algebraically independent, $\beta_2\in \overline{k(x)^F}\cap \overline{k(x_2)^F}=k$. \\

\noindent This concludes the proof of Lemma \ref{ML}.
 \end{proof}

\begin{lemm}\label{Tech}
Let $x\in F^\times\setminus F^{\times p}$ and $y\in  \overline{k(x)^F}^\times$. Write $dy=y'd x$. Then 

\begin{itemize}
\item [(i)] $ \frac{xy'}{y}=m \Rightarrow \, y\in F^{\times p}x^m $, $m\in \Z$
\item [(ii)] $y\in k(x)^\times$ and $\frac{xy'}{y}\in F^{\times p}\, \Rightarrow \, y\in F^{\times p}x^\Z$  (in particular, $x\sim_py$).
\end{itemize}
 \end{lemm}
 
 \begin{proof} For (i), just observe that  $ \frac{xy'}{y}=m$ if and only if $$d(\frac{y}{x^m})=\frac{x^my'-mx^{m-1}y}{x^{2m}}=0$$
For (ii), write $y=\frac{a}{b}A^p$ with $0\not= a,b\in k[x]$ coprime and with zeros of multiplicities at most $p-1$ and $a$ monic (if $p=0$, just impose $a,b\in k[x]$ to be coprime and $a$ monic). By assumption there exist $u,v\in k[x]$ coprime such that 
 $$x(a'b-ab')v^p=u^pab.$$
 Assume $a$ has a non-zero root $\alpha$ of multiplicity $1\leq n_\alpha\leq p-1$. Then, since $a$ and $b$ are coprime, on the left hand side, the multiplicity of $\alpha$ is congruent to  $n_\alpha-1$  (mod~$p$), while, one the right hand side, the  multiplicity of $\alpha$ is congruent to  $n_\alpha$ (mod $p$): a contradiction. This shows there exists $0\leq m\leq p-1$ such that $a=x^m$. The equation thus becomes  $$(mb-xb')v^p=u^pb$$
 or, equivalently,
 $$b(m^{\frac{1}{p}}v-u)^p=(mv^p-u^p)b=xb'v^p$$
(recall $k$ is perfect).  Again,   considering the multiplicity of a non-zero root of $b$, one sees that $b \in k^\times x^n$ for some integer $0\leq n\leq p-1$.
As a result $y=\frac{a}{b}A^p\in F^{\times p} x^\Z$ as claimed. 
\end{proof}

  \subsection{Recovering lines in $F^\times/p$, $F'^{\,\times}/p$ up to powers}  Let $k,k'$ be perfect fields of characteristic $p\geq 0$, let $F/k$, $F'/k'$ be finitely generated  regular field extensions  of transcendence degree  $\geq 2$ and let $$
\overline{\psi}: F^\times/k^\times \tilde{\rightarrow} F'^{\,\times}/k'^{\,\times}$$ be a group isomorphism preserving algebraic dependence. Write again $$\overline{\psi}:F^\times\!/p\tilde{\rightarrow}F'^{\,\times}\!/p$$ for the group isomorphism induced by  $\overline{\psi}$.

 \begin{prop}\label{Lines}There exists $m\in \Z$ satisfying (\ref{mcond}) such that for every $\overline{x}_1\not=\overline{x}_2\in F^\times/p$, $$\overline{\psi}(\frak{l}_{F^p}(\overline{x}_1,\overline{x}_2))^m= \frak{l}_{F^{' p}}(\overline{\psi}(\overline{x}_1)^m,\overline{\psi}(\overline{x}_2)^m)$$ 
\end{prop}

 \begin{proof} For simplicity, write $\overline{x}':=\overline{\psi}(\overline{x})$. We proceed in two steps.\\
 
\begin{itemize}[leftmargin=* ,parsep=0cm,itemsep=0cm,topsep=0cm]
 \item Step 1: We first show there exists $m\in \Z$ satisfying (\ref{mcond}) such that for every $\overline{x}_1\not=\overline{x}_2\in F^\times/p$, $$\overline{\psi}(\frak{l}_{F^p}(\overline{x}_1,\overline{x}_2))^m\subset \frak{l}_{F^{' p}}(\overline{x}^{'m}_1, \overline{x}^{'m}_2)$$ 
 By Lemma \ref{Shift}, one may assume $(\overline{x}_1',\overline{x}_2')\in F'^{\,\times}/k'^{\,\times}$ is a good pair.   Since $$\frak{l}_{F^p}(\overline{x}_1,\overline{x}_2)=\bigcup_{\alpha\in F^\times } F^{p\times}\frak{l}_k(\overline{x}_1,\overline{\alpha}^p\overline{x}_2).$$ 
 and for every $\alpha\in F^\times$, $(\overline{x}_1',\overline{\alpha}^p\overline{x}_2')\in F'^{\,\times}/k'^{\,\times}$ is again a good pair, it is 
  enough to prove that there exists $m\in \Z$ satisfying (\ref{mcond}) such that for every $\overline{x}_1\not=\overline{x}_2\in F^\times/p$ for which  $(\overline{x}_1', \overline{x}_2')\in F'^{\,\times}/k'^{\,\times}$ is a good pair, one has
 $$\overline{\psi}(\frak{l}_k(\overline{x}_1, \overline{x}_2))^m\subset \frak{l}_{F^{'p}}(\overline{x}_1^{'m}, \overline{x}_2^{'m}).$$
Write 
$$\frak{I}(\overline{x}_1', \overline{x}_2')_m:=\lbrace I'\in \frak{I}(\overline{x}_1', \overline{x}_2')\;|\; I'^{\,m}\in \frak{l}_{F'^{\,p}}(\overline{x}_1'^{\,m}, \overline{x}_2'^{\,m})\rbrace.$$
From Lemma \ref{ML} and Lemma \ref{LinesLem2} below, one has $$\frak{I}(\overline{x}_1', \overline{x}_2')= \lbrace \overline{x}_1', \overline{x}_2'\rbrace\bigsqcup_{m}(\frak{I}(\overline{x}_1', \overline{x}_2')_m\setminus \lbrace \overline{x}_1', \overline{x}_2'\rbrace),$$
where the union is over all $m\in \Z$ satisfying (\ref{mcond}) and for every $I'\in \frak{I}(\overline{x}_1', \overline{x}_2')_m$, one has 
 $$\frak{l}(\overline{I'},\overline{x}_i')^m\subset \frak{l}_{F'^{\,p}}(1,\overline{x}_i'^{\,m}), \; i=1,2.$$
 From Lemma \ref{LinesLem1} below, $\frak{l}_k(\overline{x}_1, \overline{x}_2)\subset  \frak{I}(\overline{x}_1,\overline{x}_2)$ so that $\overline{\psi}(\frak{l}_k(\overline{x}_1, \overline{x}_2))\subset  \frak{I}(\overline{x}'_1,\overline{x}'_2)$. Fix $I\in \frak{l}_k(\overline{x}_1, \overline{x}_2)$, $\overline{I}\not=\overline{x}_1,\overline{x}_2$ in $F^\times/p$ and let $m:=m(\overline{x}_1,\overline{x}_2, \overline{I})\in \Z$ be the unique integer satisfying (\ref{mcond}) such that $I'\in \frak{I}(\overline{x}_1', \overline{x}_2')_m$.  For $i=1,2$, one  has $\frak{l}_k(1,x_i)\subset \frak{l}(\overline{I},\overline{x}_i)$ hence 
$$\overline{\psi}(\frak{l}_k(1,x_i))^m\subset \overline{\psi}(\frak{l}(\overline{I},\overline{x}_i))^m\subset \frak{l}(\overline{I'},\overline{x}_i')^m\subset \frak{l}_{F'^{\,p}}(1,\overline{x}_i'^{\,m}).$$
By Lemma \ref{LinesLem2}, this characterizes $m$ as the unique integer satisfying (\ref{mcond}) such that 
$$\overline{\psi}(\frak{l}_k(1,x_i))^m\subset   \frak{l}_{F'^{\,p}}(1,\overline{x}_i'^{\,m}).$$
Hence  $m$ is uniquely determined by any of the two sets  $\overline{\psi}(\frak{l}_k(1,x_i))$, $i=1,2$ and, in particular,  does not depend on $I$. As a result:
$$\overline{\psi}(\frak{l}_k(\overline{x}_1,\overline{x}_2))^m\subset  \frak{l}_{F'^{\,p}}(\overline{x}_1'^{\,m},\overline{x}_2'^{\,m}).$$
In fact, $m$ does not depend on the line $\frak{l}_k(\overline{x}_1,\overline{x}_2)$ either. Indeed, let $\frak{l}_k(\overline{y}_1,\overline{y}_2)$  be any other line such that $(\overline{y}'_1,\overline{y}'_2)\in F'^{\,\times}/k'^{\,\times}$ is a good pair and let $n\in \Z$ satisfying (\ref{mcond}) be the attached integer. 
Then, necessarily, at least one of the two pairs $(\overline{x}_1',\overline{y}_1')$, $(\overline{x}_1',\overline{y}_2')$ - say $(\overline{x}_1',\overline{y}_1')$ - is a good pair; let $r\in \Z$ satisfying (\ref{mcond}) be the attached integer. Then, by considering $\overline{\psi}(\frak{l}_k(1,x_1))$, one has  $m=r$ and by considering $\overline{\psi}(\frak{l}_k(1,y_1))$, one  has  $n=r$.\\

\item Step 2: Since the situation is symmetric in $F^\times/k^\times$ and $F'^{\,\times}/k'^{\,\times}$ (here, we use that $\overline{\psi}$ preserves algebraic dependence if and only if $\overline{\psi}^{-1}$ does, by the very definition of `preserving algebraic dependence'), there exists    $m'\in \Z$ satisfying conditions (\ref{mcond}), such that for every $\overline{x}_1',\overline{x}_2'\in F'^{\,\times}/k'^{\,\times}$ with $\overline{x}_1'\not=\overline{x}_2'$ in $F'^{\,\times}/p$ one also has 
\begin{equation}\label{starprime}
\;\;\overline{\psi}^{-1}(\frak{l}_{F'^{\,p}}(\overline{x}'_1,\overline{x}'_2))^{m'}\subset \frak{l}_{F^{p}}(\overline{\psi}^{-1}(\overline{x}'_1)^{m'}, \overline{\psi}^{-1}(\overline{x}_2')^{m'})
\end{equation}
in $F^{\times}\!/p$.
As a result
$$
\frak{l}_{F^p}(\overline{x}_1,\overline{x}_2)^{mm' }=\overline{\psi}^{-1}(\overline{\psi}(\frak{l}_{F^p}(\overline{x}_1,\overline{x}_2))^{m })^{m' } 
\subset \overline{\psi}^{-1}(\frak{l}_{F'^{\,p}}(\overline{\psi}(\overline{x}_1)^{m }, \overline{\psi}(\overline{x}_2)^{m }))^{m' }
\subset  \frak{l}_{F^p}(\overline{x}_1^{mm' },\overline{x}_2^{mm' })$$

\

\noindent In particular, if $\overline{x}_1,\overline{x}_2\in F^\times/k^\times$ are such that $d x_1, d x_2$ are linearly independent over $F$, we obtain 
$$(x_1+x_2)^{mm'}= a_1^px_1^{mm' }+a_2^px_2^{mm'}\;\hbox{\rm for some}\; a_1,a_2\in F^\times
$$
hence
$$mm' (x_1+x_2)^{mm'-1}(d x_1+d x_2)=mm'(a_1^px_1^{mm'-1}d x_1+a_2^px_2^{mm' -1}d x_2)
$$ 
and $$(x_1+x_2)^{mm' -1}=a_1^px_1^{mm'-1}=a_2^px_2^{mm' -1}.$$
This forces $$mm' \equiv 1\, (mod\, p).$$

\

\noindent Now, let $\overline{x}_1,\overline{x}_2\in F^\times/k^\times$ with $\overline{x}_1\not=\overline{x}_2$ in $ F^\times/p$ and apply (\ref{starprime}) to $\overline{x}_i'=\overline{\psi}_1(\overline{x}_i)^{m}$, $i=1,2$. Using $mm' \equiv 1\, (mod\, p)$, we obtain
$$\overline{\psi}^{-1}(\frak{l}_{F'^{\,p}}(\overline{\psi}(\overline{x}_1)^m,\overline{\psi}(\overline{x}_2)^m))\subset \frak{l}_{F^p}(\overline{x}_1,\overline{x}_2)^m$$
in $F^\times\!/p$. This concludes the proof of Proposition \ref{Lines}.

\end{itemize}
\end{proof}

\begin{lemm}\label{LinesLem1} 
For every $\overline{x}_1,\overline{x}_2\in F^\times/k^\times$ such that $\overline{x}_1\not=\overline{x}_2\in F^\times/p$,  one has  $\frak{l}_k(\overline{x}_1,\overline{x}_2)\subset  \frak{I}(\overline{x}_1,\overline{x}_2)$.
 
\end{lemm}
\begin{proof}Fix $c\in k$. On the one hand
 $ x_1- cx_2= (\frac{x_1}{x_2}-c)x_2\in k(x_1/x_2)^\times x_2$
 and on the other hand, for every  $c_1,c_2\in k^\times$ such that $c=c_1/c_2$, $$ x_1- cx_2=(\frac{x_1-c_1}{x_2-c_2}-c)(x_2-c_2)\in k(\frac{x_1-c_1}{x_2-c_2})^\times (x_2-c_2).$$
 To ensure that $   x_1- cx_2 \in \frak{I}(\overline{x}_1,\overline{x}_2)$ we have to check that $x_i\not\sim_px_i-c_i$ for  $i=1,2$. Otherwise, there would be nonzero integers $a_i,b_i\in \Z$  and $\alpha_i\in \overline{k(x_i)^F}$ such that   $$\frac{(x_i-c_i)^{b_i-1}x_i^{a_i-1}(b_ix_i-a_i(x_i-c_i))}{x_i^{2a_i}}=d(\frac{(x_i-c_i)^{b_i}}{x_i^{a_i}})=0,$$
which forces $a_i=b_i=0$: a contradiction. \end{proof}
 
\begin{lemm}\label{LinesLem2} 
For every $\overline{x} \in F^\times/k^\times$ such that $ \overline{x}\not=1$ in $F^\times/p$  and for  $m\not= n\in \Z$ satisfying (\ref{mcond}), $$\frak{l}_{F^p}(1,\overline{x}^m)^n\cap \frak{l}_{F^{p}}(1,\overline{x}^n)^m\subset \{1\}\cup F^{p\times} x^{mn}.$$ 
\end{lemm}
\begin{proof} If $\frak{l}_{F^p}(1,\overline{x}^m)^n\cap \frak{l}_{F^{p}}(1,\overline{x}^n)^m\setminus \{1\}\not=\emptyset$, there exist $a,b,c,d\in F$,  
$b\neq 0\neq d$,
such that 
$$(a^p+b^px^m)^n=(c^p+d^px^n)^m.$$
Taking the logarithmic differentials and using that $dx\not= 0 $, one gets
$$ \frac{b^px^{m-1}}{a^p +b^px^m}=\frac{d^px^{n-1}}{c^p +d^px^n},$$
hence $c^pb^px^{m-1}=a^pd^px^{n-1}$, which is only possible if $m=n$ or $cb=ad=0$. But, in turn, $cb=ad=0$ is possible only if $a=c=0$.
 \end{proof}

\subsection{End of the proof of Theorem \ref{TAD}}\label{TheEnd}\textit{}\\

 \noindent From Proposition \ref{Lines} and  Lemma \ref{FTPG} $(ii)$, applied to the field extensions $F/F^p$  and $F'/F'^{\,p}$, there exist   an integer   $m\in \Z$ satisfying conditions (\ref{mcond}), and a unique field isomorphism 
 $\phi:F\tilde{\rightarrow} F'$ such that the following diagram
$$\xymatrix{F^\times\ar[r]^\phi\ar[d]&F'^{\,\times}\ar[d]\\
F^\times\!/p\ar[r]_{\overline{\phi}=\overline{\psi}^{m}}&F'^{\,\times}\!/p}$$
commutes. This concludes the proof for $p=0$. For $p>0$, one needs to work   more, using Lemma \ref{Explicit}.    \\

 \noindent Consider the group morphism 
$$\begin{tabular}[t]{llll}
$\overline{\epsilon}:$&$F^\times/k^\times$&$\rightarrow$&$F'^{\,\times p}/k'^{\,\times}$\\
&$\overline{x}$&$\rightarrow$&$\overline{\psi}(\overline{x})^{m}\overline{\phi}(\overline{x})^{-1}$.
\end{tabular}$$ 
\noindent We are to show that $m=\pm 1$ and $\overline{\epsilon}$ is trivial.\\

\noindent Fix a set-theoretic lift $\psi:F^\times\rightarrow F'^{\,\times}$ of $\overline{\psi}:F^\times/k^\times\rightarrow F'^{\,\times}/k'^{\,\times}$ and set $$\begin{tabular}[t]{llll}
$ \epsilon:$&$F^\times $&$\rightarrow$&$F'^{\,\times } $\\
&$x$&$\rightarrow$&$ \psi(x)^{m} \phi(x)^{-1}$
\end{tabular}$$

\noindent Let  $x_1,x_2\in F^\times$ be such that $(\overline{x}_1',\overline{x}_2')\in F'^{\,\times}/k'^{\,\times}$ is a good pair.  From  Lemma \ref{Explicit}, one has 
$$\overline{\psi( x_1+x_2)^{m}}=\overline{\alpha^p(\psi(x_1)^{m}+\lambda(\frac{\psi(x_1)^{Np}}{\psi(x_2)^{Np}}\psi(x_2)^{m})}$$
in $F'^{\,\times}/k'^{\,\times}$ for some $\alpha\in F'^{\,\times}$, $\lambda\in k'^{\,\times}$ and $N\in \Z$ (depending on $x_1,x_2$). Using the definition of $\epsilon$ and using that $\phi:F\rightarrow F'$ is a field homomorphism (hence is compatible with the additive structure), one has 
$$\overline{\psi( x_1+x_2)^{m}}=\overline{\epsilon(x_1+x_2)\phi(x_1+x_2)}=\overline{\frac{\epsilon(x_1+x_2)}{\epsilon(x_1)}\psi (x_1)^{m}+\frac{\epsilon(x_1+x_2)}{\epsilon(x_2)}\psi(x_2)^{m}}.$$
in $F'^{\,\times}/k'^{\,\times}$. This shows that $$\mu\alpha^p(\psi(x_1)^{m}+\lambda(\frac{\psi(x_1)^{Np}}{\psi(x_2)^{Np}}\psi(x_2)^{m})= \frac{\epsilon(x_1+x_2)}{\epsilon(x_1)}\psi(x_1)^{m}+\frac{\epsilon(x_1+x_2)}{\epsilon(x_2)}\psi(x_2)^{m}$$
for some $\mu\in k'^{\,\times}$. 

Since  $ \psi(x_1),  \psi(x_2)$ are algebraically  independent over $F'^{\,p}$ by assumption, and since $p\nmid m$, $ \psi(x_1)^{m},  \psi(x_2)^{m}$ are  linearly independent over $F'^{\,p}$ so that
$$\overline{\frac{\epsilon(x_1+x_2)}{\epsilon(x_1)}}=\overline{\alpha^p},\;\;\overline{ \frac{\epsilon(x_1+x_2)}{\epsilon(x_2)}}=\overline{\alpha^p\frac{  \psi(x_1)^{Np}}{  \psi(x_2)^{Np}}}.$$
Combining both equalities one obtains
$$\overline{\epsilon}(\overline{x}_1)=\frac{  \overline{\psi}(\overline{x}_1)^{Np}}{  \overline{\psi}(\overline{x}_2)^{Np}}\overline{\epsilon}(\overline{x}_2).$$
Now fix two primes $p'\not=p''$ distinct from $p$ and such that $p'\not\equiv 1\,(mod\, p'')$, and apply the above to $x_1,x_2^{p'}$, $x_1,x_2^{p''}$ to get
$$\overline{\epsilon}(\overline{x}_1)= \frac{  \overline{\psi}(\overline{x}_1)^{N'p}}{  \overline{\psi}(\overline{x}_2)^{N'p'p}}\overline{\epsilon}(\overline{x}_2)^{p'}$$
$$\overline{\epsilon}(\overline{x}_1)= \frac{  \overline{\psi}(\overline{x}_1)^{N''p}}{  \overline{\psi}(\overline{x}_2)^{N''p''p}}\overline{\epsilon}(\overline{x}_2)^{p''}$$
for some $N', N''\in \mathbb Z$.
Since the map $x'\rightarrow x'^{\,p}$ is injective on $F'^{\,\times}/k'^{\,\times}$, we deduce
$$(\frac{ \overline{\psi}(\overline{x}_1)^{N-N'}}{ \overline{\psi}(\overline{x}_2)^{N-N'p'}})^{p''-1}=(\frac{ \overline{\psi}(\overline{x}_1)^{N-N''}}{ \overline{\psi}(\overline{x}_2)^{N-N''p''}})^{p'-1}.$$
Since $ \overline{\psi}(\overline{x}_1)$, $ \overline{\psi}(\overline{x}_2)$ are multiplicatively independent, this forces
$$\begin{tabular}[t]{l}
$ (p''-p')N-(p''-1)N'+(p'-1)N''=0$\\
$ (p''-p')N-p'(p''-1)N'+p''(p'-1)N''=0$\\
\end{tabular}$$
Reducing modulo $p''$ and using  that $p'\not\equiv 1\, (mod\, p'')$, one sees that the matrix 
$$\left(\begin{tabular}[c]{ccc}
$ (p''-p')$&$-(p''-1)$&$(p'-1) $\\
$ (p''-p')$&$p'(p''-1)$&$p''(p'-1) $\\
\end{tabular}\right)$$
has rank $2$. Since  $(1,1,1)$ is a solution of the system above, we deduce $N=N'=N''$. This implies 
$$\overline{\epsilon}(\overline{x_2})^{(p'-1)}=\overline{\psi} (\overline{x_2})^{Np(p'-1)}, \text{ hence }\overline{\epsilon}(\overline{x}_2)=\overline{\psi} (\overline{x_2})^{Np},$$
and $$\overline{\epsilon}(\overline{x}_1)= \overline{\psi} (\overline{x}_1)^{Np}.$$ In particular, this shows  
 that $N$ does not depend on $x_2$. But it does not depend on $x_1$ either: if $y_1\in F^\times$ is another element such that $ \psi (y_1)$ is $F'/k'$-regular then either $d \psi(x_1),d\psi(y_1)$ are linearly independent over $F'$ and one applies the above with $(x_1,x_2):=(x_1,y_1)$ or one can always find $x_2\in F^\times$ such that $d\psi(x_1),d\psi(x_2)$ and $d\psi(y_1),d\psi(x_2)$ are  linearly independent over $F'$ and one applies the above with $(x_1,x_2):=(x_1,x_2)$ and $(x_1,x_2):=(y_1,x_2)$ respectively. Since, by Lemma \ref{Shift}, for every $x\in F^\times\setminus F^{\times p}$ there exists $y\in F^\times$ such that $(\overline{x},\overline{y})\in F^\times/k^\times$ is a good pair,  we deduce  $$\overline{\epsilon}(\overline{x})=\overline{\psi} (\overline{x})^{Np} , \; x\in F^\times\setminus F^{\times p}.$$
 But by multiplicativity of $\overline{\epsilon}$, $\overline{\psi}$ this also holds for $x\in F^{\times p}\setminus k^\times$ since such an $x$ can be written as $x=x_0^{p^s}$ for some $x_0\in F^\times\setminus F^{\times p}$ and integer $s\geq 1$. 
 By definition of $\overline{\epsilon}$, this means $$\overline{\psi} (\overline{x})^{m - Np}=\overline{\phi}(\overline{x}),  \; \overline{x}\in F^\times/k^\times.$$
 Since $\phi:F\tilde{\rightarrow}F'$ is a field isomorphism, this is only possible if $m - Np=\pm 1$ (otherwise, the resulting morphism of groups $\overline{\phi}:F^\times/k^\times\rightarrow F'^{\,\times}/k'^{\,\times}$ would not be surjective). Since $m $ satisfies conditions  (\ref{mcond})
 this forces $N=0$, $m=\pm 1,$ hence $$\overline{\psi} (\overline{x}) =\overline{\phi}(\overline{x})^{\pm 1},\; x\in F^\times$$  
as claimed.

\section{The fundamental theorem of projective geometry}

\begin{lemm}\label{FTPG} 
Let $k, k'$ be
fields and let $F/k$ and $F'/k'$ be field extensions. 
Let  $$\phi:F^\times/k^\times\to F'^{\,\times }/k'^{\,\times }$$ be  a group morphism. Assume that:
\begin{itemize}
\item[(i)] $\phi$ is injective and  preserves collinearity:  for any line $ L\subset F^\times/k^\times$ there is a line $L'\subset F'^{\,\times }/k'^{\,\times}$, such that $\phi(L)\subset L',$ or that
\item[(ii)] $\phi$ is an isomorphism and preserves lines:  for any line  $L\subset F^\times/k^\times$ there is a line $L'\subset F'^{\,\times}/k'^{\,\times}$, such that $\phi(L)= L'.$ 
\end{itemize}
  Assume that the image of $\phi$ is contained in no dimension  $\leq 2$ projective subspace of $F'^{\,\times}/k'^{\,\times}$. Then, in case (i) (resp., in case (ii)), there is a unique
 field  morphism  (resp. isomorphism) $\Phi: F\to F'$ 
such that the induced group morphism $\overline{\Phi}:F^\times/k^\times\to F'^{\,\times}/k'^{\,\times}$ coincides with $\phi$.\\  
 \end{lemm} 
\noindent Lemma \ref{FTPG} is elementary and well-known to experts. See for instance  \cite[Chap. II, Thm. 2.26]{Artin} for a classical formulation in the setting of finite dimensional vector spaces. For the convenience of the reader, we include here a proof in the setting of (not necessarily finite) field extensions. \\
  
\

\noindent Recall that for $x\in F^\times$ we denote by $\overline{x}$ its image $\overline{x}\in F^\times/k^\times$.

\subsection{Definition of $\Phi$}\label{PhiDef}
\subsubsection{Definition on $\lbrace 0,1\rbrace$}Set $\Phi(0)=0$, $\Phi(1)=1$.
\subsubsection{Definition on $F\setminus k$} For every $x\in F^\times\setminus k^\times$,  we have $$\overline{\phi}(\overline{1+x})\in \frak{l}_{k'}(1, \overline{\phi}(\overline{x})),$$ so that there exists a unique $\Phi(x)\in F'^{\,\times}$ such that 
$$\overline{\Phi(x)}=\overline{\phi}(\overline{x})\;\; \hbox{\rm and}\;\; \overline{\phi}(\overline{1+x})=\overline{1+\Phi(x)}.$$

\subsubsection{Definition on $k\setminus \lbrace 0,1\rbrace$}\label{PhiDefk}For every $x\in F^\times $, set $\Phi_x(0)=0$. For every $\alpha\in k^\times$,  $$\overline{\Phi(\alpha x)}=\overline{\phi}(\overline{x})=\overline{\Phi(x)},$$ so that there exists a unique $\Phi_x(\alpha)\in k'^{\,\times}$ such that 
$$\Phi(\alpha x)=\Phi_x(\alpha)\Phi(x)$$
Note that by definition $\Phi_x(1)=1$, $\Phi_x(k^\times)\subset k'^{\,\times}$.

\

\begin{lemm}  For every $x,y\in F^\times\setminus k^\times$, $\Phi_x=\Phi_y$.

\end{lemm} 

\begin{proof} $\;$
\noindent\underline{Case 1}: The elements $1$, $\Phi(x)$, $\Phi(y)$ are linearly independent over $k'$. In particular, for every $\alpha\in k^\times$,  $$\frak{l}_{k'}(\overline{\Phi(x)},\overline{\Phi(y)})\not=\frak{l}_{k'}(\overline{1+\Phi_x(\alpha)\Phi(x)},\overline{1+\Phi_y(\alpha)\Phi(y)}).$$
Since $\overline{\phi}$ preserves collinearity and
$$ \overline{(1+\alpha x)-(1+\alpha y)}=\overline{x-y}\in \frak{l}_k(\overline{x},\overline{y})\cap  \frak{l}_k(\overline{1+\alpha x},\overline{1+\alpha y})$$
we see  that
$$\lbrace \overline{\phi}(\overline{x-y})\rbrace=\frak{l}_{k'}(\overline{\Phi(x)},\overline{\Phi(y)})\cap\frak{l}_{k'}(\overline{1+\Phi_x(\alpha)\Phi(x)},\overline{1+\Phi_y(\alpha)\Phi(y)})$$
is \textit{independent of $\alpha$} while, on  the other hand, a direct computation shows that for every $\alpha\in k^\times$,
$$\lbrace \overline{\Phi_x(\alpha)\Phi(x)-\Phi_y(\alpha)\Phi(y)}\rbrace= \frak{l}_{k'}(\overline{\Phi(x)},\overline{\Phi(y)})\cap\frak{l}_{k'}(\overline{1+\Phi_x(\alpha)\Phi(x)},\overline{1+\Phi_y(\alpha)\Phi(y)}).$$
This forces $\Phi_x(\alpha)=\Phi_y(\alpha)$.\\

\noindent\underline{Case 2}:The elements $1$, $\Phi(x)$, $\Phi(y)$ are linearly dependent over $k'$. Then, by assumption,  there exists $z\in F^\times$ such that $1$, $\Phi(x)$, $\Phi(z)$ (equivalently $1$, $\Phi(y)$, $\Phi(z)$) are linearly independent over $k'$. By the above $\Phi_x=\Phi_z$ and $\Phi_y=\Phi_z$.
 
\end{proof}

\noindent In particular, for every $\alpha,\beta\in k^\times$, $x\in F^\times$, we have $$\Phi(\alpha\beta)\Phi(x)=\Phi(\alpha\beta x)=\Phi (\alpha)\Phi(\beta x)=\Phi(\alpha)\Phi(\beta)\Phi(x)$$
whence $$\Phi(\alpha\beta)=\Phi(\alpha)\Phi(\beta).
$$
Since by definition $\Phi(1)=1$, this hows that $\Phi:k^\times\rightarrow k'^{\,\times}$ is a group morphism.\\

\subsection{}We have to show that the map $\Phi:F\rightarrow F'$  defined in Subsection \ref{PhiDef} is a field morphism. By definition, we already have $\Phi(0)=0$, $\Phi(1)=1$ and for  $\alpha\in k$, $x\in F^\times\setminus k^\times$,
 $$\Phi(\alpha x)=\Phi(\alpha)\Phi(x),\;\overline{\Phi(x)}=\overline{\phi}(\overline{x}),\;  \overline{\Phi(1+x)}=\overline{\phi}(\overline{1+x})=\overline{1+\Phi(x)}.$$
 
 \
 
 \begin{lemm}\label{Add}For every $x,y\in F $   we have 
 $$\Phi(x+ y)= \Phi(x)+ \Phi(y).$$
 \end{lemm}

  \begin{proof} We may assume $x,y\not= 0$.\\
  
\noindent\underline{Case 1}:   $1$, $\Phi(x)$, $\Phi(y)$ are linearly independent over $k'$. Then
 $$\frak{l}_{k'}(\overline{1+ \Phi(x)},\overline{\Phi(y)})\not=\frak{l}_{k'}(\overline{  \Phi(x)},\overline{1+ \Phi(y)}).$$
 Since $\overline{\phi}$ preserves collinearity and
$$ \overline{1+ x +  y }\in \frak{l}_k(\overline{1+  x},\overline{y})\cap  \frak{l}_k(\overline{  x},\overline{1+  y})$$
we see  that
$$\overline{\phi}(\overline{1+  x +  y})=\frak{l}_{k'}(\overline{1+\Phi(x)},\overline{\Phi(y)})\cap \frak{l}_{k'}(\overline{  \Phi(x)},\overline{1+ \Phi(y)})$$
 while, on  the other hand, a direct computation shows that
$$\overline{1+ \Phi(x)+ \Phi(y)}=\frak{l}_{k'}(\overline{1+ \Phi(x)},\overline{\Phi(y)})\cap \frak{l}_{k'}(\overline{  \Phi(x)},\overline{1+ \Phi(y)}).$$
This forces $$\overline{1+\Phi( x +  y)}= \overline{\phi}(\overline{1+ x + y})=\overline{1+ \Phi(x)+  \Phi(y)}$$ whence $\Phi(  x +  y)= \Phi(x)+  \Phi(y)$.\\

\noindent\underline{Case 2}:  $1$, $\overline{x}$, $\overline{y}$ are all distinct in $F^\times/k^\times$ (hence $1$, $\overline{\Phi(x)}$, $\overline{\Phi(y)}$ are all distinct in $F'^{\,\times}/k'^{\,\times}$; recall $\overline{\phi}$ is injective)  but  $1$, $\Phi(x)$, $\Phi(y)$ are linearly  dependent over $k'$. Then, by assumption, there exists $z\in F^\times$ such that $1$, $\Phi(x)$, $\Phi(z)$ (equivalently $1$, $\Phi(y)$, $\Phi(z)$) are linearly independent over $k'$. Then, by the above, we have
\begin{itemize}[leftmargin=* ,parsep=0cm,itemsep=0cm,topsep=0cm]
\item[(i)] since $1$, $ \Phi(x)$, $\Phi(z)$ are linearly independent over $k'$, we have $$\Phi( x+z)= \Phi(x)+\Phi(z);$$
\item[(ii)]  since $1$, $\Phi(   x+  y)$, $\Phi(z)$ are linearly independent over $k'$, we have  $$\Phi(  x+  y+z)= \Phi( x+ y)+\Phi(z);$$
\item[(iii)] since  $1$, $\Phi(  x+z)= \Phi(x)+\Phi(z)$, $ \Phi(y)$ are linearly independent over $k'$, we have  $$\Phi(  x+   y+z)= \Phi(  x+z)+\Phi(y)= \Phi(x)+\Phi(z)+ \Phi(y).$$
\end{itemize}
Combining (ii), (iii), we obtain, again, $\Phi(\alpha x+\beta y)=\Phi(\alpha)\Phi(x)+\Phi(\beta) \Phi(y)$.\\

\noindent\underline{Case 3}:  $x=\alpha\in k^\times$, $y=x\in F^\times\setminus k^\times$. In \ref{PhiDefk} we established $$\Phi(\alpha+x)=\Phi(\alpha)\Phi(1+\alpha^{-1}x)\text{ and }\Phi(\alpha^{-1}x)=\Phi(\alpha)^{-1}\Phi(x),$$ so that it is enough to show that $\Phi(1+x)=1+\Phi(x)$. By assumption, there exists $y\in F $ such that  $1$, $\Phi(x)$, $\Phi(y)$ are linearly independent over $k'$. In particular, by Case 1, $\Phi(x+y)=\Phi(x)+\Phi(y)$. Then, the elements $1$, $\Phi(1+x)$, $\Phi(y)$ are also  linearly independent over $k'$. Hence, by Case 1, $\Phi(1+x+y)=\Phi(1+x)+\Phi(y)$. As a result,
$$\overline{\Phi(1+x+y)}=\overline{1+\Phi(x+y)}=\overline{1+\Phi(x)+\Phi(y)}$$
and 
$$\overline{\Phi(1+x+y)}=\overline{ \Phi(1+x) +\Phi(y)},$$
which forces $\Phi(1+x)=1+\Phi(x)$.\\

\noindent\underline{Case 4}:  $x=\alpha,y=\beta\in k^\times$. Let $x\in F^\times\setminus k^\times$. Then, on the one hand $$\Phi(1+\alpha x+\beta x)\stackrel{(1)}{=}\Phi(1+\alpha x)+\Phi(\beta x)\stackrel{(2)}{=}1+\Phi(\alpha x)+\Phi(\beta x)=1+(\Phi(\alpha)+\Phi(\beta))\Phi(x),$$
where (1) is by Case 2 and (2) is by Case 3. While, on the other hand, if $\alpha+\beta\not= 0$, Case 3 also yields
$$\Phi(1+\alpha x+\beta x)=1+\Phi((\alpha+\beta)x)=1+\Phi(\alpha+\beta)\Phi(x).$$
This shows $\Phi(\alpha+\beta)=\Phi(\alpha)+\Phi(\beta)$ unless $\beta=-\alpha$. For this last case, we have, by Case 3,
$$\Phi(1+\alpha x)=1+\Phi(\alpha)\Phi(x),\,\Phi(1-\alpha x)=1+\Phi(-\alpha)\Phi(x)$$ and, by Case 2, $$\Phi(2)=\Phi((1+\alpha x)+(1-\alpha x))=\Phi(1+\alpha x)+\Phi(1-\alpha x).$$ This implies 
$$2+(\Phi(\alpha)+\Phi(-\alpha))x=\Phi(2)\in k'$$
hence $\Phi(\alpha)+\Phi(-\alpha)=0$.
 \end{proof}
 
 \

 \begin{coro}For every $x,y\in F^\times $   we have 
 $$\Phi(x  y)= \Phi(x)  \Phi(y).$$
 \end{coro}

\begin{proof}  By Lemma \ref{Add}, we have $$\Phi(x(1+y))=\Phi(x)+\Phi(xy)\text{ and }\Phi(1+y)=1+\Phi(y).$$ Also, since $\overline{\phi}:F^\times/k^\times\rightarrow F'^{\,\times}/k'^{\,\times}$ is a group morphism, we have $$\overline{\Phi(xy)}=\overline{\phi(xy)}=\overline{\phi}(\overline{x})\overline{\phi}(\overline{y})=\overline{\Phi(x)\Phi(y)}$$
and 
\begin{multline*}
$$\overline{\Phi(x)+\Phi(xy)}=\overline{\Phi(x(1+y))}=\overline{\phi(\overline{x}\overline{(1+y)})}=\overline{\phi}(\overline{x})\overline{\phi}(\overline{1+y})=\\=\overline{\Phi(x)(1+\Phi(y))}=\overline{\Phi(x)+\Phi(x)\Phi(y)}.
\end{multline*}
This forces $\Phi(xy)=\Phi(x)\Phi(y)$ as claimed.

\end{proof}

\subsection{End of the proof} At this stage, we have shown that there exists a field morphism $\Phi:F\rightarrow F'$ such that 
\begin{itemize} 
\item $\Phi(k)\subset k'$; 
\item  the induced group morphism $\overline{\Phi}:F^\times/k^\times\rightarrow F'^{\,\times}/k'^{\,\times}$ coincides with $$\overline{\phi}~:~F^\times/k^\times~\rightarrow~F'^{\,\times}/k'^{\,\times}.$$
\end{itemize}
It remains to prove the unicity and  Part (ii): $\Phi:F\rightarrow F'$ is an isomorphism if   $\overline{\phi}:F^\times/k^\times\rightarrow F'^{\,\times}/k'^{\,\times}$ is.

\subsubsection{Unicity}Let $\Psi:F\rightarrow F'$ be another field morphism such that    $\Psi(k)\subset k'$ and  $\overline{\Psi}=\overline{\phi}$. Then, for every $x\in F^\times $ there exists a unique $\lambda_x\in k'^{\,\times}$ such that $\Phi(x)=\lambda_x\Psi(x)$. But necessarily we have $\lambda_1=1$ and, if $x\in F^\times\setminus k^\times$, so that $\Phi(x),\Psi(x)\in F'^{\,\times}\setminus k'^{\,\times}$, and
  $$1+\lambda_x\Psi(x)=1+\Phi(x)=\Phi(1+x)=\lambda_{1+x}\Psi(1+x)=\lambda_{1+x}(1+\Psi(x)).$$
This shows $\lambda_x=\lambda_{1+x}=1$. If $\alpha\in k^\times$ and $x\in F^\times\setminus k^\times$, we have $$\lambda_\alpha \Psi(\alpha)\lambda_x\Psi( x)=\Phi(\alpha)\Phi(x)=\Phi(\alpha x)=\lambda_{\alpha x}\Psi(\alpha x).$$
Since $\lambda_{\alpha x}=\lambda_x=1$ by the above, this forces $\lambda_\alpha=1$ as well.\\

\noindent This concludes the proof of Lemma \ref{FTPG} (i).

\subsubsection{Part (ii)}Lemma \ref{FTPG} (ii)  follows formally from the existence and uniqueness assertion in Lemma \ref{FTPG} (i). Indeed, since $\phi^{-1}:F'^{\,\times}/k'^{\,\times}\rightarrow F^\times/k^\times$ is also a group isomorphism which preserves lines, there exists a unique field morphism $\Psi:F'\rightarrow F$ such that the induced group morphism $\overline{\Psi}:F'^{\,\times}/k'^{\,\times}\rightarrow F^\times/k^\times$ coincides with $\phi^{-1}:F'^{\,\times}/k'^{\,\times}\rightarrow F^\times/k^\times$. Applying again this argument with the identity morphisms of $F^\times/k^\times$, $F'^{\,\times}/k'^{\,\times}$, one gets $\Phi\circ \Psi=Id_{F'}$, $\Psi\circ \Phi= Id_F$.

\bibliographystyle{plain}
\bibliography{MilnorSubmission}
 
\end{document}